\documentclass{article} 
\usepackage[usenames,dvipsnames]{color}
\usepackage{amsmath} 
\usepackage{amssymb} 
\usepackage{bm}
\usepackage{amsthm}
\usepackage{CJK}
\usepackage{indentfirst}
\usepackage{tikz}
\usepackage{amsmath,amsfonts,amsthm,mathrsfs}
\usepackage{hyperref}
\numberwithin{equation}{section}
\usepackage{geometry}

\usepackage{esint} 

\numberwithin{equation}{section}
\newtheorem{theorem}{Theorem}
\newtheorem{lemma}{Lemma}

\newtheorem{corollary}{Corollary}
\newtheorem{definition}{Definition}
\newtheorem{remark}{Remark}

\newtheorem{assumption}{Assumption}
\usepackage{algorithm,algpseudocode}
\usepackage{amsthm}

\def\NN{\mathbb{N}}

\def\RR{\mathbb{R}}

\def\ZZ{\mathbb{Z}}
\def\YY{\mathbb{Y}}
\def\CC{\mathbb{C}}

\def\SS{\mathbb{S}}

\def\mal{\max\limits}
\def\mil{\min\limits}
\def\sul{\sum\limits}

\def\rr{ r}
\def\ss{ s}
\def\res{{\rm Res }}
\def\re{{\rm Re }}
\def\im{{\rm Im }}

\def\lt{\left}
\def\rt{\right}

\def\mH{\mathcal{H}^\rr(\SS^d)}

\title{Sharp Approximation on the Sphere by\\ Linearized Shallow Networks with Analytic Activations}
\author{Jia Li\thanks{Peking University, Beijing 100871, P. R. China}\and Tong Mao\thanks{Shenzhen Loop Area Institute, Shenzhen 518000, P. R. China}\and Jinchao Xu\thanks{King Abdullah University of Science and Technology, Thuwal 23955, Saudi Arabia}}
\date{}

\begin{document}
\maketitle

\begin{abstract}
We study the spectral behavior and approximation properties of linearized shallow networks with analytic activation functions on the unit sphere $\mathbb S^d$. Motivated by earlier spherical approximation results for analytic zonal functions \cite{mhaskar1999approximation}, we focus on the spectral information needed to obtain Sobolev-norm estimates together with explicit control of the outer coefficients. We first derive an asymptotic formula, expressed explicitly in terms of the residues, for the ultraspherical coefficients of a real-analytic function whose relevant singularities are a conjugate pair of simple poles on the imaginary axis. In particular, the formula gives a matching lower bound along the relevant high-degree subsequences and identifies the exponential decay scale and sign oscillation. It applies directly to $\tanh$, the logistic sigmoid, and the rational activations considered here, while the corresponding estimate for $\arctan$ is obtained through differentiation.

For such activation functions and quasi-uniform points $\{\theta_j^\ast\}_{j=1}^n \subset \mathbb S^d$, functions $f \in \mathcal H^r(\mathbb S^d)$ satisfying the appropriate spectral compatibility condition can be approximated by a linear combination
\[
f_n(\eta) = \sum_{j=1}^n a_j \sigma(\theta_j^\ast \cdot \eta),
\]
with the sharp Sobolev rate
\[
\|f - f_n\|_{\mathcal H^s(\mathbb S^d)} \lesssim n^{-\frac{r-s}{d}} \|f\|_{\mathcal H^r(\mathbb S^d)}.
\]
The approximation theorem places the $\mathcal H^s$ norm on the left-hand side and simultaneously provides an explicit normalized $\ell^2$ bound on the coefficients. For spectrally compatible analytic target functions, the same construction further gives the stated analytic convergence rate. Thus, the analysis connects spectral asymptotics and lower bounds at the degrees where the ultraspherical coefficients do not vanish with Sobolev-norm approximation and coefficient control for fixed analytic activation functions.
\end{abstract}

\section{Introduction}

Deep neural networks are constructed by composing multiple layers of nonlinear transformations, each of which can be viewed as a shallow neural network. In this sense, shallow networks serve as the basic building blocks of deep architectures. Understanding their structure and approximation properties is therefore a natural first step toward a theoretical understanding of deep learning models.

Given an activation function $\sigma:\mathbb R \to \mathbb R$, we define the class of shallow neural networks with $n$ neurons on $\Omega\subset\RR^d$ by
\begin{equation}\label{shallow}
    \Sigma_{n}^\sigma(\Omega)
    =
    \left\{
        x\mapsto\sum_{j=1}^n a_j \sigma\big(w_j\cdot x + b_j\big)
        :
        w_j \in \mathbb R^d,\;
        a_j,b_j \in \mathbb R
    \right\}.
\end{equation}
Alternatively, writing $\displaystyle\tilde x = \binom{x}{1} \in \mathbb R^{d+1}$ and $\displaystyle\theta_j = \binom{w_j}{b_j} \in \mathbb R^{d+1}$, the network can be expressed in the compact form $\sul_{j=1}^n a_j \sigma(\theta_j \cdot \tilde x)$.

On the unit sphere $\mathbb S^d \subset \mathbb R^{d+1}$, the corresponding spherical neural network class is
\[
\Sigma_{n}^\sigma(\mathbb S^d)
=
\left\{
    \eta\mapsto\sum_{j=1}^n a_j \sigma(\theta_j \cdot \eta)
    :
    \theta_j \in \mathbb S^d,\;
    a_j \in \mathbb R
\right\},
\]
where $\eta \in \mathbb S^d$. Spherical domains arise naturally when the data or the underlying symmetries possess an intrinsic directional or rotational structure. This has motivated rotation-equivariant spherical convolutional networks and graph-based spherical architectures \cite{cohen2018spherical,jiang2019spherical,defferrard2019deepsphere}, together with theoretical results on the approximation of spherical Sobolev functions, generalization for spherical-data classification, and approximation of nonlinear functionals defined on spaces of spherical functions \cite{fang2020theory,feng2023generalization,yang2026spherical}. A substantial approximation theory has also been developed for scattered-data interpolation and spherical basis functions \cite{jetter1999error,hubbert2004lp}, as well as for localized frames and constructive polynomial approximation on the sphere \cite{narcowich2006localized,sloan2012filtered}. The present work contributes to this literature by establishing sharp approximation rates for linearized shallow networks with analytic activations and fixed quasi-uniform directions on $\mathbb S^d$.


The approximation theory of shallow neural networks has a long history, beginning with the universal approximation phenomenon established around the late 1980s and early 1990s. In particular, qualitative density results show that single-hidden-layer networks can approximate broad classes of functions for a variety of activation functions (e.g., sigmoidal). Representative theorems include the universal approximation results of \cite{cybenko1989approximation,funahashi1989approximate,hornik1989multilayer} and the sharp characterization in \cite{leshno1993multilayer} stating that non-polynomial activations yield density in $\mathcal C(\Omega)$ for compact $\Omega$.

Beyond qualitative density, quantitative approximation rates for shallow networks have been investigated in several complementary function space settings. A recurring conclusion is that, for target functions in Barron or variation-type spaces, shallow networks can achieve the ``Monte-Carlo'' rate
\begin{equation}\label{half_related}
  \inf_{f_n\in \Sigma_n^\sigma(\Omega)}
  \|f - f_n\|_{\mathcal{L}^2(\Omega)}
  =\mathcal{O}(n^{-\frac{1}{2}}).
\end{equation}
This rate is typically obtained through sampling or selection arguments rooted in Maurey-type techniques; see, for example,
\cite{pisier1981remarques,maurey1973type,Jones1992,barron1993universal,barron1994approximation,devore1996some,makovoz1998uniform,kurkova1,kurkova2,lewicki2004approximation,temlyakov2008greedy,Barron2008,E2019population,E2019barron,E2020representation}. 
Comprehensive overviews of these developments can be found in 
\cite{pinkus1999approximation,DeVore2021,konyagin2018some}.

A complementary direction seeks sharper rates under additional structural assumptions on the target function. In particular, for Barron-type, spectral Barron, and related function classes, one can improve upon \eqref{half_related} and obtain dimension-dependent gains of the form
\begin{equation*}
  \inf_{f_n\in \Sigma_n^\sigma(\Omega)}
  \|f - f_n\|_{\mathcal{L}^2(\Omega)}
  = \mathcal{O}\!\left(n^{-\frac{1}{2}-\frac{\alpha}{d}}\right),
\end{equation*}
for some $\alpha>0$. Such refined estimates have been established in
\cite{makovoz1996random,bach2017breaking,klusowski2018approximation,xu2020finite,siegel2022high,siegel2022sharp,siegel2022optimal,mhaskar2023tractability,ma2022uniform,meng2022new,mao2023rates,yang2024optimal}. 
These results highlight that additional regularity or structural constraints on the target function can significantly enhance the approximation efficiency of shallow neural networks beyond the classical $n^{-1/2}$ regime.

Finally, although depth can provide additional advantages for compositional target functions \cite{poggio2017theory}, sharp smoothness-dependent approximation rates already arise at the level of shallow networks. In particular, for any $r>0$, functions in the Sobolev space $\mathcal{H}^r(\Omega)$ can be approximated by shallow networks with sufficiently smooth activations at the optimal rate \cite{mhaskar1993approximation,petrushev1998approximation,pinkus1999approximation} (see also \cite{devore1989optimal} for optimality)
\begin{equation}\label{eqn_smooth_nonlinear1}
    \inf_{f_n\in \Sigma_n^\sigma(\Omega)}
    \|f - f_n\|_{\mathcal{L}^2(\Omega)}
    = \mathcal{O}\!\left(n^{-\frac{r}{d}}\right).
\end{equation}
If the target function is analytic, then shallow networks with smooth or analytic activations achieve substantially faster convergence under suitable analyticity assumptions \cite{mhaskar1996neural},
\begin{equation}\label{eqn_smooth_nonlinear2}
    \inf_{f_n\in \Sigma_n^\sigma(\Omega)}
    \|f - f_n\|_{\mathcal{L}^2(\Omega)}
    = \mathcal{O}\!\left(\rho^{-\sqrt[d]{n}}\right).
\end{equation}
These results highlight that even in the absence of depth, shallow architectures can attain approximation rates that sharply reflect the regularity of the target function.

Earlier work \cite{mhaskar1999approximation} established sharp $\mathcal L^p$ Sobolev approximation orders on $\SS^d$ for zonal-function networks generated by analytic activation functions. In the notation used here, if the ultraspherical coefficients satisfy
\begin{equation}\label{eqn_intro_coeff_rate}
|\widehat{\sigma}(m)|\simeq m^\alpha\rho^{-m},
\end{equation}
then, for a quasi-uniform set $\{\theta_j^*\}_{j=1}^n\subset\mathbb S^d$, the linearized space
\begin{equation}\label{linearV}
L_n^\sigma=L_n^\sigma(\{\theta_j^*\}_{j=1}^n)=\left\{\eta\mapsto\sum_{j=1}^na_j\sigma(\theta_j^*\cdot\eta):a_j\in\mathbb R\right\}
\end{equation}
satisfies the sharp approximation estimate
\begin{equation}\label{eqn_MNW}
    \inf_{f_n\in L_n^\sigma}\|f-f_n\|_{\mathcal L^p(\mathbb S^d)}\lesssim\|f\|_{\mathcal{W}^{r,p}(\SS^d)}n^{-\frac{r}{d}}.
\end{equation}
Thus, under suitable coefficient assumptions, the analytic-activation setting already admits the sharp $\mathcal L^p(\SS^d)$ order $n^{-r/d}$. The question considered below is whether an $\mathcal H^s$ error estimate and an explicit outer-coefficient bound can be obtained simultaneously for concrete activations whose coefficients may change sign or vanish on one parity class.

More recently, \cite{liu2025achieving} showed that shallow ReLU$^k$ networks with fixed inner parameters can attain the sharp approximation rate $\mathcal O\!\left(n^{-r/d}\right)$ for functions in the Hilbert Sobolev spaces $\mathcal H^r(\Omega)$ and $\mathcal H^r(\mathbb S^d)$, provided that $r\le\frac{d+2k+1}{2}$. In the spherical setting, there exist coefficients $a_1,\dots,a_n$ satisfying
$$\Big(n\sul_{j=1}^na_j^2\Big)^{\frac12}\lesssim\|f\|_{\mathcal{H}^{r}(\SS^d)}n^{\frac{d+2k+1}{2d}-\frac{r}{d}}$$
such that \cite{liu2025achieving}
\begin{equation}\label{eqn_rate_lin_reluk}
\|f-f_n\|_{\mathcal H^s(\mathbb S^d)}\lesssim\|f\|_{\mathcal{H}^{r}(\SS^d)}n^{-\frac{r-s}{d}},\qquad 0\le s\leq k,~s\le r\le\frac{d+2k+1}{2}.
\end{equation}
This result realizes the approximation estimate with an $\mathcal H^s(\SS^d)$ norm on the left-hand side and simultaneously controls the outer coefficients, although for ReLU$^k$ the admissible Sobolev range is limited by the finite smoothness of the activation.

A key ingredient in this analysis is the precise behavior of the ultraspherical or Gegenbauer coefficients of the activations. These coefficients determine which spherical harmonic degrees can be reproduced by the ridge dictionary and control the spectral multipliers appearing in the approximation construction. Therefore, for a smooth activation $\sigma$, the first question is not directly whether the space $L_n^\sigma$ has a certain approximation power, but whether the ultraspherical coefficients of $\sigma$ can be understood with sufficient precision.

For analytic activations, this question is governed by their complex singularities. It is classical that analytic continuation to a Bernstein ellipse yields exponential decay of orthogonal-polynomial coefficients, with the largest admissible ellipse, and hence the exponential rate, determined by the nearest complex singularities; see, for example, \cite{elliott1964evaluation,zhao2013sharp}. For ultraspherical or Gegenbauer expansions, contour-integral representations involving a hypergeometric kernel were developed in \cite{cantero2012rapid} and subsequently used to obtain sharp coefficient estimates in \cite{wang2016optimal}. These contour representations provide the starting point for our spectral analysis. For the approximation construction below, the relevant information is a two-sided asymptotic formula with explicit residue dependence. Such a formula supplies a matching lower bound at the degrees where the coefficients do not vanish, retains the sign oscillation, and identifies the relevant parity subsequences.

Accordingly, our first contribution, Theorem~\ref{thm:tanh-coeff-rate}, establishes, on the relevant spectral degrees,
\begin{equation}
|\widehat{\sigma}(m)|\simeq m^{-\frac{d-1}{2}}\rho_*^{-m},
\end{equation}
for real-analytic functions whose nearest singularities are a conjugate pair of simple poles on the imaginary axis, where $\rho_*>1$ is determined by their locations. In particular, the two-sided estimate includes a matching lower bound for the ultraspherical coefficient magnitude along the relevant high-degree subsequences. The theorem also identifies the residue-dependent oscillation, while its activation-specific consequences in Theorem~\ref{thm_coeff} determine the parity patterns for the concrete examples. It applies directly to $\tanh$, the logistic sigmoid, and the rational activations considered below, while the corresponding estimate for $\arctan$ follows from a differentiation identity.

The preceding coefficient estimate also verifies, for the activations considered here, the coefficient assumptions associated with the $\mathcal L^p$ estimate \eqref{eqn_MNW} in \cite{mhaskar1999approximation}. Our objective is different: we seek an estimate of the form \eqref{eqn_rate_lin_reluk}, with the error measured in $\mathcal H^s(\SS^d)$, together with explicit normalized $\ell^2$ control of the outer coefficients. Inspired by approximation ideas in \cite{mhaskar1999approximation} and by the Hilbert-space estimates in \cite{liu2025achieving}, we develop a Hilbert-space construction that simultaneously achieves these two properties for analytic activations whose ultraspherical coefficients decay exponentially and may oscillate in sign. More precisely, for any $r>0$ and $0\le s\le r$, and every spectrally compatible target $f\in\mathcal H^r(\mathbb S^d)$, there exist coefficients $a_1,\dots,a_n$ with
$$\Big(n\sul_{j=1}^na_j^2\Big)^{\frac12}\lesssim\rho^{C_3\sqrt[d]{n}}\|f\|_{\mathcal{L}^2(\SS^d)}$$
such that
\[
\|f-f_n\|_{\mathcal H^s(\mathbb S^d)}\lesssim\|f\|_{\mathcal{H}^{r}(\SS^d)}n^{-\frac{r-s}{d}}.
\]
For the activations treated explicitly, spectral compatibility reduces to the parity conditions stated in Corollary~\ref{cor_appr_rate}. We also establish the stated analytic convergence estimate for analytic target functions. When $s=0$, the Sobolev order agrees with \eqref{eqn_MNW}; the additional conclusions are the full $\mathcal H^r$-to-$\mathcal H^s$ estimate and the simultaneous normalized $\ell^2$ coefficient bound for the oscillatory and parity-restricted spectra under consideration.

The two contributions are connected as follows: the residue analysis gives the spectral asymptotics, lower bounds along the relevant subsequences, and the nonvanishing pattern. This information then enters the Hilbert-space construction for fixed quasi-uniform directions, which gives an $\mathcal H^s$ error estimate together with explicit normalized $\ell^2$ control of the outer coefficients.

The closest comparison for analytic activation functions is \cite{mhaskar1999approximation}, which already establishes sharp Sobolev approximation orders in the corresponding $\mathcal L^p$ setting under exponentially decaying ultraspherical coefficients. Our focus is different: we derive precise coefficient asymptotics for activations with oscillatory and parity-restricted ultraspherical coefficients and use them to obtain an $\mathcal H^s(\SS^d)$ error estimate together with explicit normalized $\ell^2$ control of the outer coefficients.

Other constructions provide complementary points of comparison. The approach of \cite{petrushev1998approximation} reduces the multidimensional problem to one-dimensional approximation and employs a product-type discretization combining quadrature directions on $\mathbb S^{d-1}$ with nodes in $[-1,1]$. The manifold framework of \cite{mhaskar2020kernel} obtains the order $\mathcal O(n^{-r/d})$ under monotonicity and fast-decay assumptions on the ultraspherical coefficients, including rapid decay of the ratio $\widehat{\sigma}(B^*m)/\widehat{\sigma}(m)$ for some $B^*>1$. The sign-oscillating and partially vanishing spectra considered here are treated instead through the two-sided asymptotics and lower bounds in Theorem~\ref{thm:tanh-coeff-rate}. Thus, the Sobolev order agrees with the earlier analytic-activation result, while the present refinement concerns the $\mathcal H^s$ error, coefficient control, and the coefficient structure used to obtain them.

The remainder of the paper is organized accordingly. Section~\ref{sec_preli} collects the harmonic-analytic and geometric preliminaries and explains the spectral compatibility condition linking the two parts of the analysis. Section~\ref{sec_mainresult} presents the ultraspherical coefficient asymptotics and lower bounds along the relevant subsequences, followed by the Sobolev-norm approximation, coefficient control, and analytic approximation results. Section~\ref{sec_proof} proves these results in the same order. Section~\ref{sec_concl} concludes with several directions for further study.


\section{Preliminaries}\label{sec_preli}

This section collects the preliminaries for the two parts of the paper. The spherical harmonic decomposition supports the coefficient analysis and identifies the spectral modes generated by an activation. The geometry of quasi-uniform points and positive quadrature formulas is then used in the approximation construction.

We first introduce the notation used in this paper. We use $\NN_0=\{0,1,2,\ldots\}$ for the nonnegative integers and $\NN=\{1,2,3,\ldots\}$ for the positive integers. Following \cite{xu1992iterative}, we use $\gtrsim$, $\lesssim$, and $\simeq$ to express upper and lower bounds up to constant factors. When we write
$$
f(x)\gtrsim g(x),\quad g(x)\lesssim h(x),\quad h(x)\simeq k(x),
$$
it means that there exist positive constants $c_1,c_2,c_3,c_4$ independent of $x$ such that
$$
f(x)\ge c_1 g(x),\quad g(x)\le c_2 h(x),\quad c_3 h(x)\le k(x)\le c_4 h(x).
$$

\subsection{Spherical harmonics, ultraspherical and Gegenbauer polynomials}
We use standard notation from spherical harmonic analysis; see \cite{dai2013approximation,stein1971introduction}. Let $\SS^d=\{\eta\in\RR^{d+1}:|\eta|=1\}$, let $\omega_d$ be its surface area, and write
\[
\fint_{\SS^d}f(\eta)\,d\eta:=\frac{1}{\omega_d}\int_{\SS^d}f(\eta)\,d\eta.
\]
The corresponding $\mathcal L^2(\SS^d)$ inner product is denoted by $\langle\cdot,\cdot\rangle_{\mathcal L^2(\SS^d)}$, and the geodesic distance is $d(\eta,\theta)=\arccos(\eta\cdot\theta)$.

For $m\ge0$, let $\YY_m$ be the space of spherical harmonics of degree $m$ and choose an orthonormal basis $\{Y_{m,\ell}\}_{\ell=1}^{N(m)}$ with respect to the normalized measure. Its dimension is $N(0)=1$ and
\[
N(m)=\frac{2m+d-1}{m}\binom{m+d-2}{d-1},\qquad m\ge1.
\]
The mutually orthogonal spaces $\YY_m$ yield the decomposition $\mathcal L^2(\SS^d)=\bigoplus_{m\ge0}\YY_m$. Thus, for $f\in\mathcal L^2(\SS^d)$,
\[
f=\sum_{m=0}^\infty\Pi_mf,\qquad \Pi_mf=\sum_{\ell=1}^{N(m)}\widehat f(m,\ell)Y_{m,\ell},\qquad \widehat f(m,\ell)=\langle f,Y_{m,\ell}\rangle_{\mathcal L^2(\SS^d)},
\]
where the series converges in $\mathcal L^2(\SS^d)$.

The reproducing kernel of $\YY_m$ is zonal. We therefore normalize the degree-$m$ ultraspherical polynomial $p_m$ by the addition formula
\begin{equation}\label{eqn:sum_Y_nl}
p_m(\eta\cdot\theta)=\sum_{\ell=1}^{N(m)}Y_{m,\ell}(\eta)Y_{m,\ell}(\theta),\qquad \eta,\theta\in\SS^d.
\end{equation}
In particular, $p_m(1)=N(m)$. These polynomials are orthogonal in $\mathcal L^2_{w_d}([-1,1])$, where
\[
\langle f,g\rangle_{w_d}:=\int_{-1}^1f(t)g(t)(1-t^2)^{\frac{d-2}{2}}\,dt,\qquad \|f\|_{\mathcal L^2_{w_d}([-1,1])}:=\langle f,f\rangle_{w_d}^{1/2},
\]
and their normalization in \eqref{eqn:sum_Y_nl} gives
\begin{equation}\label{eqn:Pn_normalization}
\|p_m\|_{\mathcal L^2_{w_d}([-1,1])}^2=\frac{\omega_d}{\omega_{d-1}}N(m).
\end{equation}
Consequently, every $\sigma\in\mathcal L^2_{w_d}([-1,1])$ has the expansion
\begin{equation}\label{eqn_ultraspherical_expansion_tanh}
\sigma=\sul_{m=0}^\infty\widehat\sigma(m)p_m,\qquad \widehat\sigma(m)=\frac{\langle\sigma,p_m\rangle_{w_d}}{\|p_m\|_{\mathcal L^2_{w_d}([-1,1])}^2},
\end{equation}
with convergence in $\mathcal L^2_{w_d}([-1,1])$.

We measure Sobolev regularity and spectral analyticity through the same spherical harmonic decomposition.
\begin{definition}
\begin{enumerate}
    \item For $\rr\ge0$, define
    \[
    \mH:=\{f\in\mathcal L^2(\SS^d):\|f\|_{\mH}<\infty\},
    \]
    where
    \begin{equation}\label{eqn:Sob_norm_Parseval}
    \|f\|_{\mH}^2=\sul_{m=0}^\infty\sul_{\ell=1}^{N(m)}(m^{2\rr}+1)|\widehat f(m,\ell)|^2.
    \end{equation}
    We identify $\mathcal H^0(\SS^d)$ with $\mathcal L^2(\SS^d)$.
    \item For $\rho>1$, define
    \[
    \mathrm{Hol}_\rho(\SS^d):=\{f\in\mathcal L^2(\SS^d):\|f\|_{\mathrm{Hol}_\rho(\SS^d)}<\infty\},
    \]
    where
    \begin{equation}
    \|f\|_{\mathrm{Hol}_\rho(\SS^d)}^2=\sul_{m=0}^\infty\sul_{\ell=1}^{N(m)}(\rho^{2m}+1)|\widehat f(m,\ell)|^2.
    \end{equation}
\end{enumerate}
\end{definition}

Throughout the paper we work with the ultraspherical polynomials ${p_m}$ introduced above, together with our normalization. However, the auxiliary results we quote from \cite{wang2016optimal} are formulated for a differently normalized family called Gegenbauer polynomials. For the reader's convenience, we briefly recall their definition and record the conversion between the two normalizations. All quoted statements are then translated into our notation through this relation.

For a fixed $\lambda>-\frac{1}{2}$ and $\lambda\neq 0$, the Gegenbauer polynomials are orthogonal over the interval $[-1, 1]$ with respect to the weight function $w(t)=(1-t^2)^{\lambda-\frac{1}{2}}$ and
\begin{equation}
    \int_{-1}^1 (1-t^2)^{\lambda-\frac{1}{2}}C_m^{(\lambda)}(t)C_n^{(\lambda)}(t)\,dt = h_m^{(\lambda)}\delta_{mn}, 
\end{equation} 
where $\delta_{mn}$ is the Kronecker delta and 
\begin{equation}
    h_m^{(\lambda)} = \frac{2^{1-2\lambda}\pi}{\Gamma(\lambda)^2}\frac{\Gamma(m+2\lambda)}{\Gamma(m+1)(m+\lambda)}, \quad \lambda\neq 0.
\end{equation}
When $d\geq 2$, by selecting $\lambda=\frac{d-1}{2}$, the general Gegenbauer polynomials are ultraspherical polynomials up to multiplicative constants. 

To distinguish the coefficients associated with the two bases, denote 
$$\widehat{\sigma}_\ast(m) = \frac{\lt< \sigma,C_m^{(\lambda)}\rt>_{w_d}}{\|C_m^{(\lambda)}\|_{\mathcal{L}_{w_d}^2([-1,1])}^2},\quad \widehat{\sigma}(m) = \frac{\lt< \sigma,p_m\rt>_{w_d}}{\|p_m\|_{\mathcal{L}_{w_d}^2([-1,1])}^2}.$$
The relation between $\widehat{\sigma}_\ast(m)$ and $\widehat{\sigma}(m)$ is
\begin{equation}\label{eqn:coeff-relation}
\begin{split}
 q_{m,d} &:= \frac{\widehat{\sigma}(m)}{\widehat{\sigma}_\ast(m)} = \frac{\|C_m^{(\lambda)}\|_{\mathcal{L}_{w_d}^2([-1,1])}}{\|p_m\|_{\mathcal{L}_{w_d}^2([-1,1])}} \\
     &= \left( \frac{2^{2-d}\pi}{\Gamma(\frac{d-1}{2})^2}\frac{\Gamma(m+d-1)}{\Gamma(m+1)(m+\frac{d-1}{2})} \bigg/ \frac{\omega_d}{\omega_{d-1}}\frac{2m+d-1}{m}\binom{m+d-2}{d-1} \right)^{1/2} \\
     &= \lt(\frac{2^{1-d}\,\pi\,\Gamma(d)\,\omega_{d-1}}{\omega_d}\rt)^{\frac{1}{2}}\frac{1}{\Gamma(\frac{d-1}{2})}\,\lt(m+\frac{d-1}{2}\rt)^{-1}.
     \end{split}
\end{equation}

Now we record the precise link between the coefficient and approximation parts of the paper. Combining \eqref{eqn_ultraspherical_expansion_tanh} and \eqref{eqn:sum_Y_nl}, every $f_n\in L_n^\sigma$ has the $\mathcal{L}^2$-spherical expansion
\begin{equation}\label{eqn_fn_expans}
    \begin{split}
        f_n(\eta)=&\sul_{j=1}^na_j\sigma(\theta_j^*\cdot\eta)=\sul_{j=1}^na_j\sul_{m=0}^\infty\widehat{\sigma}(m)p_m(\theta_j^*\cdot\eta)=\sul_{j=1}^na_j\sul_{m=0}^\infty\widehat{\sigma}(m)\sul_{\ell=1}^{N(m)}Y_{m,\ell}(\theta_j^*)Y_{m,\ell}(\eta)\\
        =&\sul_{m=0}^\infty\widehat{\sigma}(m)\sul_{\ell=1}^{N(m)}\Big(\sul_{j=1}^na_jY_{m,\ell}(\theta_j^*)\Big)Y_{m,\ell}(\eta).
    \end{split}
\end{equation}
It follows immediately that if $\Pi_mf\neq0$ for some $m$ satisfying $\widehat{\sigma}(m)=0$, then
$$\|f-f_n\|_{\mathcal{H}^s(\SS^d)}^2\geq\|\Pi_mf-\Pi_mf_n\|_{\mathcal{H}^s(\SS^d)}^2=\|\Pi_mf\|_{\mathcal{H}^s(\SS^d)}^2>0.$$
Hence, a necessary condition for
\begin{equation*}
    \lim\limits_{n\to\infty}\inf\limits_{f_n\in L_n^\sigma}\|f-f_n\|_{\mathcal H^s(\mathbb S^d)}=0
\end{equation*}
is that $\Pi_mf=0$ whenever $\widehat{\sigma}(m)=0$. Thus the nonvanishing set of the activation coefficients is the interface between the two parts of the analysis: the first part identifies this set for concrete activations, and the second approximates targets supported on it. We formulate the required compatibility as follows.
\begin{assumption}[Spectral compatibility]\label{assum_nece}
$\Pi_mf=0$ for all $m\notin E_\sigma$, where
\begin{equation}
    E_\sigma:=\Big\{m:~\widehat{\sigma}(m)\neq0\Big\}.
\end{equation}
\end{assumption}

For the activations used below, this condition takes a simple parity form. For example,
\begin{enumerate}
    \item if $\sigma=\tanh$, then $E_\sigma\subset2\NN_0+1$, and $f_n$ is necessarily an odd function. Any function with
    $$f(-x)+f(x)\neq0$$
    cannot be approximated by $L_n^\sigma$.
    \item if $\sigma(t)=\frac{1}{1+t^2}$, then $E_\sigma\subset2\NN_0$, and $f_n$ is necessarily an even function. Any function with
    $$f(-x)-f(x)\neq0$$
    cannot be approximated by $L_n^\sigma$.
\end{enumerate}

\subsection{Scattered points and polynomial quadrature}\label{subsec:prop_ultraspherical}
In this subsection, we fix a finite collection of pairwise distinct points 
$\{\theta_j^*\}_{j=1}^n \subset \mathbb{S}^{d}$, which are assumed to be scattered on the sphere. 
We measure the density of this set by its mesh norm
\begin{equation}
h = \max_{\eta\in\SS^d}\min_{1\le j\le n}d(\eta,\theta_j^*).
\end{equation}
This quantity represents the largest geodesic distance from any point on $\SS^d$ to the nearest sampling point.

Related Marcinkiewicz--Zygmund and quadrature constructions on general manifolds were developed in \cite{filbir2011marcinkiewicz}, while localized polynomial operators based on scattered spherical data and quadrature were studied in \cite{leGiaMhaskar2009localized}.

Positive quadrature formulas with polynomial exactness and controlled weights for sufficiently dense scattered point sets on $\SS^d$ follow from \cite[Corollary~4.4]{narcowich2006localized}. We use the following consequence.

\begin{lemma}[{\cite[Corollary~4.4]{narcowich2006localized}}]\label{lem:quadrature}
There exist nonnegative weights $\tau_1,\dots,\tau_n$ with $\tau_j\lesssim h^d$, and a constant $C_2$ independent of $n$ and $h$, such that
\begin{equation}\label{eqn:quadrature_formula}
    \fint_{\SS^d}p(\eta)d\eta=\sul_{j=1}^n\tau_jp(\theta_j^*),\quad \forall p\in\mathbb{P}_{2J}(\SS^d),
\end{equation}
where $J=\lfloor C_2h^{-1}\rfloor$. Moreover, if the separation distance satisfies
\begin{equation}
    \mil_{i\neq j}d(\theta_i^*,\theta_j^*)\gtrsim h,
\end{equation}
then
\begin{equation}
    \tau_j\simeq h^d\simeq n^{-1},\qquad j=1,\dots,n.
\end{equation}
All of the corresponding constants are independent of $n$ and $\{\theta_j^*\}_{j=1}^n$.
\end{lemma}

In this paper, we assume the collection $\{\theta_j^*\}_{j=1}^n \subset \SS^d$ is quasi-uniform. In this case, $h\simeq n^{-\frac{1}{d}}$.
\begin{definition}[Quasi-uniform]
    Let $d\in\NN$. A set of points $\{\theta_j^*\}_{j=1}^n\subset\SS^d$ is said to be quasi-uniform if
    \begin{equation}\label{eqn:quasiuniform}
        \mal_{\theta\in\SS^d}\min\limits_{1\leq j\leq n}d(\theta,\theta_j^*)\lesssim\min\limits_{i\neq j}d(\theta_i^*,\theta_j^*).
    \end{equation}
    The corresponding constants are independent of $n$.
\end{definition}

\section{Main Results}\label{sec_mainresult}

This section presents the two connected parts of the analysis. Section~\ref{subsec_coeff} derives ultraspherical coefficient asymptotics from a conjugate pair of simple poles and then specializes them to several smooth activations. In particular, the resulting two-sided estimates provide the lower bounds needed to invert the corresponding nonzero spectral multipliers. Inspired by earlier approximation constructions for quasi-uniform parameters on the sphere, Section~\ref{subsec:approx-rates} develops a Hilbert-space argument that yields an approximation estimate with an $\mathcal H^s$ norm on the left-hand side together with explicit normalized $\ell^2$ control of the outer coefficients.

\subsection{Ultraspherical coefficient asymptotics from conjugate imaginary poles}\label{subsec_coeff}

Analyticity in a Bernstein ellipse gives an exponential upper bound for ultraspherical coefficients, but the later approximation construction requires a matching lower bound at the degrees where the coefficients do not vanish, together with the leading oscillation and the relevant nonvanishing subsequences, because the low-frequency matching step divides by the activation coefficients. The following theorem derives the corresponding asymptotic formula with explicit residue dependence for a conjugate pair of simple poles on the imaginary axis.
\begin{theorem} \label{thm:tanh-coeff-rate}
Suppose that $\sigma$ is real on $[-1,1]$ and analytic inside and on the Bernstein ellipse
$$\mathcal{E}_R:= \{\frac{1}{2}(R e^{i\theta} + R^{-1}e^{-i\theta})\mid \theta\in[0,2\pi)\}$$
except for two simple poles $z_\pm$ in its interior. Assume that $z_\pm$ lie on $i\RR$, that $\im \ z_+>0$, and that the residues of $\sigma$ at $z=z_\pm$ satisfy $\res(\sigma,z_\pm)\neq 0$. Then, as $m\to+\infty$,  
\begin{equation} \label{eqn:coeff-asymptotic}
    \widehat{\sigma}(m)=\mathcal O(R^{-m})-I_+(m)-I_-(m)
\end{equation}
where
\begin{equation*}
     I_+(m)+I_-(m)=\re\big(\res(\sigma,z_+)(-i)^{m+1}\big) \rho_\ast^{-m}\left\{\begin{array}{ll}
        C_1\,\rho_\ast^{-1}\int_0^1t^m(1-t)^{\lambda-1}(1+\frac{t}{\rho_\ast^2})^{\lambda-1}\,dt ,&\ d\geq 2, \\[6pt]
        4(\rho_\ast + \rho_\ast^{-1})^{-1}, &\ d=1.
    \end{array}\right.
\end{equation*}
with
$\rho_\ast = \lt|z_\pm + \sqrt{z_\pm^2-1}\rt| = |z_\pm| + \sqrt{|z_\pm|^2+1}$, $\lambda=\frac{d-1}{2}$ and
$$C_1=C_1(d)=4\lt(\frac{2^{1-d}\,\pi\,\Gamma(d)\,\omega_{d-1}}{\omega_d}\rt)^{\frac{1}{2}}\frac{1}{\Gamma(\frac{d-1}{2})}.$$
Moreover, there exists $M=M(\sigma)$ such that for $m\in F_{\sigma,M}$
\begin{equation} \label{eqn:coeff-bound}
    |\widehat{\sigma}(m)|\simeq m^{-\frac{d-1}{2}}\rho_\ast^{-m},
\end{equation}

where
\begin{equation}\label{eqn:coeff-nonzero-set}
    F_{\sigma,M} := \left\{ \begin{aligned}
        \{m:~m\geq M,~m\equiv 1\mod 2\}, &\quad \res(\sigma,z_\pm)\in \RR,\\
        \{m:~m\geq M,~m\equiv 0\mod 2\}, &\quad \res(\sigma,z_\pm)\in i\RR,\\
        \{m:~m\geq M\}, &\quad  \res(\sigma,z_\pm)\notin \RR\cup i\RR.
    \end{aligned}\right.
\end{equation}

\end{theorem}

The quantity $\rho_\ast$ is the Bernstein-ellipse parameter associated with the two poles, whereas the term $\mathcal O(R^{-m})$ comes from a larger ellipse and is therefore of strictly smaller exponential order. The phase of the residue determines which parity subsequence carries the leading term. The theorem is deliberately stated for an imaginary conjugate pair, which is the configuration shared by the activations considered here; more general pole configurations would require a separate analysis of possible interference among several oscillatory residue contributions.

Theorem~\ref{thm:tanh-coeff-rate} applies directly to the first four activations listed below. The $\arctan$ case does not follow directly from the pole theorem, because $\arctan$ has logarithmic branch points at $\pm i$; it is obtained instead from the rational case by differentiating its Gegenbauer expansion.
\begin{theorem}\label{thm_coeff}
    \begin{enumerate}
        \item Let $\sigma=\tanh$, then $E_\sigma=2\NN_0+1$, and
    \begin{equation}
    \begin{split}
        &(-1)^{\frac{m-1}{2}}\widehat{\sigma}(m)\simeq m^{-\frac{d-1}{2}}\rho_1^{-m},\quad m\in E_\sigma,
    \end{split}
\end{equation}
where $\rho_1=\frac{\pi}{2}+\sqrt{\frac{\pi^2}{4}+1}$.
\item Let $\sigma(t)=\frac{1}{1+e^{-t}}$, then $E_\sigma=\{0\}\cup(2\NN_0+1)$, and
    \begin{equation}
        \begin{split}
            &(-1)^{\frac{m-1}{2}}\widehat{\sigma}(m)\simeq m^{-\frac{d-1}{2}}\rho_2^{-m},\quad m\in 2\NN_0+1,
        \end{split}
    \end{equation}
    where $\rho_2=\pi+\sqrt{\pi^2+1}$.
\item Let $\sigma(t) = \frac{1}{t^2+1}$, then $E_\sigma=2\NN_0$, and
    \begin{equation}\label{eqn_coef_est_3}
        \begin{split}
            &(-1)^{\frac{m}{2}}\widehat{\sigma}(m)\simeq m^{-\frac{d-1}{2}}\rho_3^{-m},\quad m\in E_\sigma,
        \end{split}
    \end{equation}
where $\rho_3=1+\sqrt{2}$.
\item Let $\sigma(t) = \frac{t}{t^2+1}$, then $E_\sigma=2\NN_0+1$, and
    \begin{equation}
        \begin{split}
            &(-1)^{\frac{m-1}{2}}\widehat{\sigma}(m)\simeq m^{-\frac{d-1}{2}}\rho_3^{-m},\quad m\in E_\sigma,
        \end{split}
    \end{equation}
        \item Let $\sigma(t)=\arctan(t)$,
        then $E_\sigma=2\NN_0+1$, and
        \begin{equation}\label{eqn_arctan_coeff}
        \begin{split}
            &(-1)^{\frac{m-1}{2}}\widehat{\sigma}(m)\simeq m^{-\frac{d+1}{2}}\rho_3^{-m},\quad m\in E_\sigma.
        \end{split}
    \end{equation}
    \end{enumerate}
\end{theorem}

\subsection{Sobolev and analytic approximation rates}\label{subsec:approx-rates}

The use of fixed quasi-uniform directions is inspired by approximation ideas in \cite{mhaskar1999approximation}, while the positive quadrature input is taken from \cite{narcowich2006localized}; related constructive ideas can also be found in \cite{mhaskar2006weighted}. We develop a Hilbert-space construction adapted to the exponentially decaying and possibly sign-oscillating ultraspherical coefficients considered here. The resulting construction provides both an approximation estimate in $\mathcal H^s$ and simultaneous explicit normalized $\ell^2$ control of the outer coefficients.

We now turn from the coefficient structure of particular activations to the approximation mechanism. The next theorem is formulated under an abstract two-sided estimate for the nonzero ultraspherical coefficients. Its lower-bound component permits inversion of the nonzero low-frequency coefficients, while its upper-bound component controls the high-frequency tail. For quasi-uniform neuron directions, the theorem realizes the order $\mathcal O(n^{-\frac{r-s}{d}})$ with the error measured in $\mathcal H^s$ and with the coefficient bound \eqref{eqn_coef_bound_thm}; the following theorem gives the corresponding analytic convergence estimate.

\begin{theorem}\label{thm_main_sob}
    Let $d,n\in\NN$, let $0\leq s\leq r$, and let $\{\theta_j^*\}_{j=1}^n\subset\SS^d$ be quasi-uniform. Suppose that $\sigma\in\mathcal{L}_{w_d}^2([-1,1])$ and that there exist $\rho>1$ and $\alpha\in\RR$ such that
    \begin{equation}\label{eqn_decay_thm1}
        \widehat{\sigma}(m)^2\simeq m^{2\alpha}\rho^{-2m},\qquad m\in E_\sigma,
    \end{equation}
    Then, for every $f\in\mathcal{H}^r(\SS^d)$ satisfying Assumption~\ref{assum_nece}, there exist $a_1,\dots,a_n\in\RR$ with
    \begin{equation}\label{eqn_coef_bound_thm}
        \Big(n\sul_{j=1}^na_j^2\Big)^{\frac12}\lesssim\rho^{C_3\sqrt[d]{n}}\|f\|_{\mathcal{L}^2(\SS^d)}
    \end{equation}
    such that
    \begin{equation}\label{eqn:sob_rate}
        \Bigl\|f-\sul_{j=1}^n a_j\sigma(\theta_j^*\cdot\circ)\Bigr\|_{\mathcal{H}^\ss(\SS^d)}\lesssim n^{-\frac{r-s}{d}}\|f\|_{\mathcal{H}^r(\SS^d)}.
    \end{equation}
    All the corresponding constants are independent of $n,\{\theta_j^*\}_{j=1}^n$, and $f$.
\end{theorem}

Theorem \ref{thm_main_sob} indicates that the approximation rate improves systematically with the Sobolev smoothness of the target function. 
This naturally leads to the analytic regime, where stronger regularity yields the faster convergence rate stated in the following theorem.

\begin{theorem}\label{thm:appr_rate_sph}
    Let $d,n,s,\alpha,\sigma,\rho$ and $\{\theta_j^*\}_{j=1}^n\subset\SS^d$ be as in Theorem~\ref{thm_main_sob}, and let $\rho_0\in(1,\rho]$. Then, for every $f\in \mathrm{Hol}_{\rho_0}(\SS^d)$ satisfying Assumption~\ref{assum_nece}, there exist $a_1,\dots,a_n\in\RR$ with
    \begin{equation}
        \Big(n\sul_{j=1}^na_j^2\Big)^{\frac12}\lesssim\rho^{C_3\sqrt[d]{n}}\|f\|_{\mathcal{L}^2(\SS^d)}
    \end{equation}
    such that
    \begin{equation}\label{eqn:rate_nonuniform}
        \Bigl\|f-\sul_{j=1}^n a_j\sigma(\theta_j^*\cdot\circ)\Bigr\|_{\mathcal{H}^\ss(\SS^d)}\lesssim \rho_0^{-C_4\sqrt[d]{n}}\lt\|f\rt\|_{\mathrm{Hol}_{\rho_0}(\SS^d)}.
    \end{equation}
    Here $C_4$ depends only on $d$, and the constant implicit in $\lesssim$ is independent of $n,\{\theta_j^*\}_{j=1}^n$, and $f$. 
\end{theorem}

Theorem~\ref{thm:appr_rate_sph} complements Theorem~\ref{thm_main_sob} by showing that, in the analytic regime, the same linearized construction achieves a rate governed by the analyticity radius.

The coefficient estimates in Theorem~\ref{thm_coeff} verify the hypothesis \eqref{eqn_decay_thm1} for the concrete activations and thereby connect the two parts of the analysis. We obtain the following consequences.

\begin{corollary}\label{cor_appr_rate}
    Let $d,n,s,r$, $\{\theta_j^*\}_{j=1}^n$ be as in Theorem \ref{thm_main_sob}.
    \begin{enumerate}
    \item Let $\displaystyle\sigma(t)=\tanh(t)$, then \eqref{eqn:sob_rate} holds for odd functions $f\in\mathcal{H}^r(\SS^d)$, \eqref{eqn:rate_nonuniform} holds for odd functions $f\in \mathrm{Hol}_{\rho_0}(\SS^d)$ with $\rho_0<\frac{\pi}{2}+\sqrt{\frac{\pi^2}{4}+1}$.
    \item Let $\displaystyle\sigma(t)=\frac{1}{1+e^{-t}}$, then \eqref{eqn:sob_rate} holds for odd functions $f\in\mathcal{H}^r(\SS^d)$, \eqref{eqn:rate_nonuniform} holds for odd functions $f\in \mathrm{Hol}_{\rho_0}(\SS^d)$ with $\rho_0<\pi+\sqrt{\pi^2+1}$.
    \item Let $\sigma$ be one of the following functions:
    \begin{equation*}
        \frac{t}{1+t^2},\quad\arctan(t),
    \end{equation*}
    then \eqref{eqn:sob_rate} holds for odd functions $f\in\mathcal{H}^r(\SS^d)$, \eqref{eqn:rate_nonuniform} holds for odd functions $f\in \mathrm{Hol}_{\rho_0}(\SS^d)$ with $\rho_0<1+\sqrt{2}$.
    \item Let $\displaystyle\sigma(t)=\frac{1}{1+t^2},$
    then \eqref{eqn:sob_rate} holds for even functions $f\in\mathcal{H}^r(\SS^d)$, \eqref{eqn:rate_nonuniform} holds for even functions $f\in \mathrm{Hol}_{\rho_0}(\SS^d)$ with $\rho_0<1+\sqrt{2}$.
    \item Let $\displaystyle\sigma(t)=\frac{1+t}{1+t^2},$
    then \eqref{eqn:sob_rate} holds for all $f\in\mathcal{H}^r(\SS^d)$, \eqref{eqn:rate_nonuniform} holds for all $f\in \mathrm{Hol}_{\rho_0}(\SS^d)$ with $\rho_0<1+\sqrt{2}$.
    \end{enumerate}
    All the corresponding constants are independent of $n,\{\theta_j^*\}_{j=1}^n$, and $f$.    
\end{corollary}

In \cite{liu2025achieving}, the sharp Sobolev approximation rate $\mathcal{O}(n^{-\frac{r}{d}})$ was also established for linearized shallow ReLU$^k$ networks on both Euclidean domains and spheres, in a more general setting. However, on the sphere, the estimate
$$
\inf_{f_n\in L_n^k}
\|f-f_n\|_{\mathcal{L}^2(\SS^d)}
\lesssim
n^{-\frac{r}{d}}
\|f\|_{\mathcal{H}^r(\SS^d)}
$$
holds only for $r\le \frac{d+2k+1}{2}$. This phenomenon, referred to as \emph{saturation} \cite{mao2025sharp}, reflects an intrinsic limitation imposed by the regularity of the activation: higher regularity $k$ permits sharp approximation up to larger values of $r$, but does not extend beyond this threshold. More broadly, converse and saturation theories for Gaussian networks and radial basis-function interpolation were developed in \cite{mhaskar2004gaussian,mhaskar2010bernstein,schaback2002inverse}.

For the analytic activations treated here, Theorem~\ref{thm_main_sob} yields the sharp order $\mathcal{O}(n^{-\frac{r}{d}})$ for arbitrary $r>0$ on the Sobolev scale. In the analytic estimate, however, the upper bound furnished by Theorem~\ref{thm:appr_rate_sph} is limited by the exponential decay scale of the activation coefficients, which is determined by the nearest poles through $\rho_\ast$. The present upper bound does not improve once the target analyticity radius exceeds the corresponding critical scale, suggesting a limitation set by the analyticity scale of the activation. A matching approximation lower bound would be needed to interpret this limitation as a genuine saturation phenomenon.


The deterministic theorem can also be combined with standard geometric estimates for independent uniform points on $\SS^d$. In this setting, the resulting model is commonly called a \emph{random feature method} or an \emph{extreme learning machine}; see, for example, \cite{rahimi2007random,rahimi2008weighted,huang2006extreme,bach2017breaking}. As in the random-direction argument of \cite[Theorem~6.2]{liu2025achieving}, one extracts a quasi-uniform subset with high probability and then applies the deterministic estimate for quasi-uniform directions. The following corollary records the resulting logarithmically modified $\mathcal H^s$ rate for the analytic activations considered here.

\begin{corollary}\label{thm:random_linear}
Let $0\le \ss\le \rr$, let $\{\theta_j\}_{j=1}^n$ be i.i.d.\ uniform samples from $\SS^d$, and let $\sigma$ be one of the activation functions in Corollary~\ref{cor_appr_rate} and
$$L_n^\sigma=L_n^\sigma(\{\theta_j\}_{j=1}^n)=\left\{\eta\mapsto\sum_{j=1}^na_j\sigma(\theta_j\cdot\eta):a_j\in\mathbb R\right\}.$$
Then for any $\delta\in(0,1)$ and any function $f\in\mathcal{H}^r(\SS^d)$ with the same parity property as in Corollary~\ref{cor_appr_rate}, with probability at least $1-\delta$,
\begin{equation}\label{eqn:rate_random_constrained_inf}
\inf_{f_n\in L_n^\sigma}
\|f-f_n\|_{\mathcal{H}^\ss(\SS^d)}
\ \lesssim\
\Big(\frac{n}{\log(n/\delta)}\Big)^{-\frac{\rr-\ss}{d}}\,
\|f\|_{\mathcal{H}^\rr(\SS^d)}.
\end{equation}
Consequently, the expected error estimate satisfies
\begin{equation}\label{eqn:rate_expectation_constrained}
\mathbb{E}_n\Bigl[\inf_{f_n\in L_n^\sigma}
\|f-f_n\|_{\mathcal{H}^\ss(\SS^d)}\Bigr]
\ \lesssim\
\Big(\frac{n}{\log n}\Big)^{-\frac{\rr-\ss}{d}}\,
\|f\|_{\mathcal{H}^\rr(\SS^d)}.
\end{equation}
\end{corollary}

\begin{proof}
    Standard covering-number and concentration arguments imply that for i.i.d.\ uniform points $\{\theta_j\}$,
\begin{equation}\label{eq:mesh_norm_random}
\sup\limits_{\eta\in\SS^d}\mil_{1\leq j\leq n}d(\eta,\theta_j) \ \lesssim\ \Big(\frac{\log(n/\delta)}{n}\Big)^{1/d}
\end{equation}
with probability at least $1-\delta$. On this event, there exists a quasi-uniform subset of $\{\theta_j\}_{j=1}^n$ with mesh size $\displaystyle\simeq\Big(\frac{\log(n/\delta)}{n}\Big)^{1/d}$ and cardinality $\displaystyle\simeq\frac{n}{\log(n/\delta)}$. Theorem~\ref{thm_main_sob} and Corollary~\ref{cor_appr_rate} then give \eqref{eqn:rate_random_constrained_inf}.
\end{proof}

\section{Proofs of the main results}\label{sec_proof}

\subsection{Proofs of the ultraspherical coefficient estimates}\label{subsec:legendre-expansion}
We first prove the coefficient results in Subsection~\ref{subsec_coeff}. The argument starts from the contour representation of ultraspherical coefficients derived in \cite{cantero2012rapid} and subsequently rederived in \cite{wang2016optimal}. After translating that representation to our normalization, we deform the contour to isolate the two pole contributions. The scattered-point geometry used in the approximation theorems plays no role in this part.
\begin{lemma} \label{lem:contour-integral-repr}
    Suppose that $\sigma$ is analytic inside and on the Bernstein ellipse $\mathcal{E}_\rho:= \{\frac{1}{2}(\rho e^{i\theta} + \rho^{-1}e^{-i\theta})\mid \theta\in[0,2\pi)\}$. Then, for every $d\geq 2$ and $m\geq 0$, 
    \begin{equation}
        \widehat{ \sigma }(m) = \frac{c_{m,d}}{i\pi}\oint_{\mathcal{E}_\rho} \frac{ \sigma (z)}{(z+\sqrt{z^2-1})^{m+1}} {}_2F_1\left[\begin{matrix} m+1,\, 1-\lambda; \\ m+\lambda+1; \end{matrix}\, \frac{1}{(z+\sqrt{z^2-1})^2}\right]\,dz,
    \end{equation}
    where $\lambda=\frac{d-1}{2}$ and $\sqrt{z^2-1}$ denotes the branch on $\CC\setminus [-1,1]$ such that $\lim_{|z|\to\infty}|z+\sqrt{z^2-1}|=\infty$. Moreover,
    \begin{equation} \label{eqn:c_md-def}
        c_{m,d} = \frac{\Gamma(\lambda)\Gamma(m+1)}{\Gamma(m+\lambda)}\cdot q_{m,d}, 
    \end{equation}
    and $\Gamma(z)$ is the gamma function. The Gauss hypergeometric function ${}_2F_1$ is defined by
    \begin{equation} \label{eqn:gauss-def}
        {}_2F_1\left[\begin{matrix} a_1,\, a_2; \\ b; \end{matrix}\, z\right]=\sum_{k=0}^\infty\frac{(a_1)_k(a_2)_k}{(b)_k}\frac{z^k}{k!}, 
    \end{equation} 
    where $(z)_k$ denotes the Pochhammer symbol defined by $(z)_k = z(z+1)\cdots(z+k-1)$ for $k \geq 1$ and $(z)_0 = 1$.
\end{lemma}
\begin{proof}
    Applying \cite[Theorem 3.2]{wang2016optimal} and relation \eqref{eqn:coeff-relation} proves the lemma.
\end{proof}

\begin{lemma}[{\cite[Theorem 2.2.1]{andrews1999special}}] \label{lem:gauss-euler-repr}
    For all $\lambda>0$, the Gauss hypergeometric function admits the Euler integral representation: 
    \begin{equation}\label{eqn:gauss-euler-repr}
        {}_2 F_1\left[\begin{matrix} m+1,\, 1-\lambda; \\ m+1+\lambda; \end{matrix}\, z\right] = \frac{\Gamma(m+\lambda+1)}{\Gamma(m+1)\Gamma(\lambda)}\int_{0}^1 t^m(1-t)^{\lambda-1}(1-zt)^{\lambda-1}\,dt,
    \end{equation}
    where the power function $w^{\lambda-1}$ in the integral is understood in the principal branch, with branch cut along $(-\infty,0]$ in the $w$-plane.
\end{lemma}

\begin{proof}[Proof of Theorem \ref{thm:tanh-coeff-rate}]
    We split the analysis into two cases. 
    
    \textbf{Case 1:} $d\geq 2$. In this case, $\lambda=\frac{d-1}{2}\neq 0$.
    
    The only singularities of $\sigma(z)$ in $\mathcal{E}_R$ are $z_\pm$. Choose $1<\rho<\rho_\ast$. Then $\sigma$ is analytic in $\mathcal{E}_\rho$, and Lemma~\ref{lem:contour-integral-repr} gives

    \begin{equation}
        \widehat{\sigma}(m) = \frac{c_{m,d}}{i\pi}\oint_{\mathcal{E}_\rho} \frac{\sigma(z)}{(z+\sqrt{z^2-1})^{m+1}} {}_2F_1\left[\begin{matrix} m+1,\, 1-\lambda; \\ m+\lambda+1; \end{matrix}\, \frac{1}{(z+\sqrt{z^2-1})^2}\right]\,dz.
    \end{equation}
    
    Because the integrand is analytic in $\mathcal{E}_R\setminus\bigl([-1,1]\cup\{z_\pm\}\bigr)$, we may deform the contour without changing the value of the integral. Deforming $\mathcal{E}_\rho$ to the larger ellipse $\mathcal{E}_R$ (positively oriented), 
    and denoting by $\mathcal{C}_\pm$ two sufficiently small positively oriented circles centered at $z_\pm$, we obtain

    $$\widehat{\sigma}(m) = I_R(m) - I_+(m) - I_-(m),$$
    where $I_R(m)$ denotes the integral on $\mathcal{E}_R$, and $I_\pm (m)$ denotes the integral on $\mathcal{C}_\pm$.

    \begin{figure}[H]
        \centering
        \includegraphics[width=0.5\linewidth]{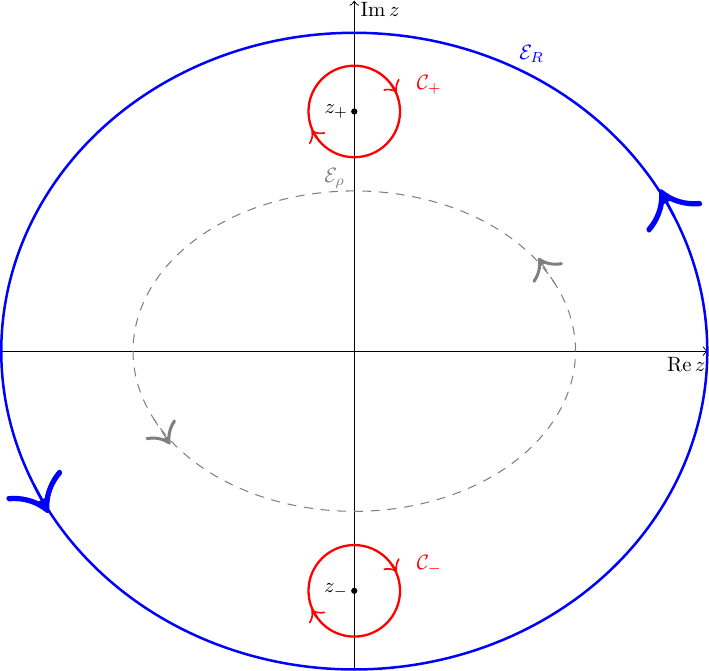}
        \caption{Ellipse Contour}
        \label{fig:ellipse-contour}
    \end{figure}

    For $I_R(m)$, since the function $z+\sqrt{z^2-1}$ transforms $\mathcal E_R$ into $\mathcal C_R = \{|z|=R\}$, by Lemma \ref{lem:gauss-euler-repr} we get
    \begin{equation}\label{eqn:I_R(m)-point-estimate}
    \begin{split}
        &\quad\left|\frac{ \sigma (z)}{(z+\sqrt{z^2-1})^{m+1}} {}_2F_1\left[\begin{matrix} m+1,\, 1-\lambda; \\ m+\lambda+1; \end{matrix}\, \frac{1}{(z+\sqrt{z^2-1})^2}\right]\right|\\
        & \leq \sup_{z\in\mathcal{E}_R}| \sigma (z)|\,R^{-m-1}\frac{\Gamma(m+\lambda+1)}{\Gamma(m+1)\Gamma(\lambda)}\int_{0}^1 t^m(1-t)^{\lambda-1}\left|1-\frac{t}{(z+\sqrt{z^2-1})^2}\right|^{\lambda-1}\,dt \\
        &\leq \sup_{z\in\mathcal{E}_R}| \sigma (z)|\,R^{-m-1}\cdot\max\left( (1+R^{-2})^{\lambda-1}, (1-R^{-2})^{\lambda-1}\right),\quad \forall z\in \mathcal{E}_R,
    \end{split}
    \end{equation}
    where the last inequality follows from the identity: $$\int_0^1 t^m(1-t)^{\lambda-1}\,dt =\frac{\Gamma(m+1)\Gamma(\lambda)}{\Gamma(m+\lambda+1)}.$$ 
    Then we have
    \begin{equation}\label{eqn:I_R(m)}
        |I_R(m)|\lesssim \sup_{z\in\mathcal{E}_R}| \sigma (z)|\cdot c_{m,d}\,R^{-m}\lesssim \sup_{z\in\mathcal{E}_R}| \sigma (z)|\cdot m^{-\frac{d-1}{2}}\,R^{-m}
    \end{equation}
    where the last inequality uses Stirling's formula together with \eqref{eqn:c_md-def} and \eqref{eqn:coeff-relation} to give
    \begin{equation}\label{eqn:c_md-estimate}
        c_{m,d}\simeq m^{-\lambda}=m^{-\frac{d-1}{2}}.
    \end{equation} 

    For $I_\pm (m)$, the residue theorem gives

    \begin{equation}
        \begin{split}
            I_\pm (m) &= \frac{c_{m,d}}{i\pi}\oint_{\mathcal{C}_\pm} \frac{\sigma(z)}{(z+\sqrt{z^2-1})^{m+1}} {}_2F_1\left[\begin{matrix} m+1,\, 1-\lambda; \\ m+\lambda+1; \end{matrix}\, \frac{1}{(z+\sqrt{z^2-1})^2}\right]\,dz \\
            &= \frac{2c_{m,d}\,\res(\sigma,z_\pm)}{\rho_\ast^{m+1}(\pm i)^{m+1}} {}_2F_1\left[\begin{matrix} m+1,\, 1-\lambda; \\ m+\lambda+1; \end{matrix}\, \frac{-1}{\rho_\ast^{2}}\right].
        \end{split}
    \end{equation}

    Since $\sigma$ is real analytic, it follows that $z_+=-z_-$ and $\res(\sigma,z_+) = \overline{\res(\sigma,z_-)}$. Hence $I_+(m)=\overline{I_-(m)}$, and therefore
    \begin{equation}\label{eqn:coeff-pm}
    \begin{split}
        I_+(m) + I_-(m) &= 4c_{m,d}\re\big(\res(\sigma,z_+)(-i)^{m+1}\big)\, \rho_\ast^{-(m+1)} {}_2F_1\left[\begin{matrix} m+1,\, 1-\lambda; \\ m+\lambda+1; \end{matrix}\, \frac{-1}{\rho_\ast^{2}}\right] \\
        &= C_1\,\rho_\ast^{-(m+1)}\,\re\big(\res(\sigma,z_+)(-i)^{m+1}\big) \int_0^1 t^m(1-t)^{\lambda-1}(1+\frac{t}{\rho_\ast^2})^{\lambda-1}\,dt,
    \end{split}
    \end{equation}
    where the last equality uses \eqref{eqn:c_md-def}, \eqref{eqn:coeff-relation} and Lemma \ref{lem:gauss-euler-repr}. 

    Combining this identity with the estimate of $I_R(m)$ in \eqref{eqn:I_R(m)}, we obtain
    \begin{equation} \label{eqn:coeff-asymptotic-1}
    \begin{split}
        \widehat{\sigma}(m) &= I_R(m) - I_+(m) - I_-(m) \\ 
        &=-C_1\,\rho_\ast^{-(m+1)}\,\re\big(\res(\sigma,z_+)(-i)^{m+1}\big) \int_0^1 t^m(1-t)^{\lambda-1}(1+\frac{t}{\rho_\ast^2})^{\lambda-1}\,dt+ \mathcal O(R^{-m}),
    \end{split}
    \end{equation}
    which completes the proof of \eqref{eqn:coeff-asymptotic} in Case 1.

    Combining \eqref{eqn:coeff-asymptotic}, \eqref{eqn:coeff-nonzero-set}, and the following estimate obtained using Stirling's formula:
    \begin{align*}
        &\int_{0}^1 t^m(1-t)^{\lambda-1}(1+\frac{t}{\rho_\ast^2})^{\lambda-1}\,dt \\
        =&\  \frac{\Gamma(m+1)\Gamma(\lambda)}{\Gamma(m+\lambda+1)}\cdot\frac{\Gamma(m+\lambda+1)}{\Gamma(m+1)\Gamma(\lambda)}\int_{0}^1 t^m(1-t)^{\lambda-1}(1+\frac{t}{\rho_\ast^2})^{\lambda-1}\,dt \\
        \simeq &\ m^{-\lambda},
    \end{align*}
    we prove \eqref{eqn:coeff-bound} in Case 1.

    \textbf{Case 2:} $d=1$. In this case, $p_m(t)$ are normalized Chebyshev polynomials $$p_m(t) = \left\{\begin{aligned}
        2\cos(m\arccos t), &\quad m\geq 1, \\
        1, &\quad m=0,
        \end{aligned}\right.,\quad\|p_m\|^2_{\mathcal{L}^2_{w_1}([-1,1])}=\left\{\begin{aligned}
        2\pi, &\quad m\geq 1, \\
        \pi, &\quad m=0,
    \end{aligned}\right.$$
    and for $m\geq 1$, $\widehat\sigma(m)$ can be calculated as
    \begin{equation}
    \begin{split}
        \widehat{\sigma}(m) &= \frac{1}{\pi}\int_{-1}^1 \sigma(t)\cos(m\arccos t)(1-t^2)^{-\frac{1}{2}}\,dt \qquad (t=\cos x)\\
        &= \frac{1}{\pi}\int_0^\pi \sigma(\cos x)\cos(m x)\,dx = \frac{1}{2\pi}\int_{-\pi}^\pi \sigma(\cos x)e^{imx}\,dx
    \end{split}
    \end{equation}

    The analyticity assumption on $\sigma$ implies that the integrand is analytic on $$\{w\in \CC:\im \,w\in[0,\log R], \re\,w\in[-\pi,\pi]\}\setminus \{w_\pm=\pm\frac{\pi}{2}+i\log\rho_\ast\},$$
    where $w_\pm$ are the isolated singularities. We therefore deform the integral to the counterclockwise rectangular contour with vertices $-\pi,\pi, \pi+i\eta, -\pi+i\eta$, where $\eta=\log R$, together with two small circles $\mathcal{C}_\pm$ centered at $w_\pm$.

    By periodicity, the integrals over the two vertical edges cancel, and we obtain
    $$\widehat{\sigma}(m)= \frac{1}{2\pi} I_0(m) =\frac{1}{2\pi}\lt( I_\eta(m) + I_+(m) + I_-(m)\rt),$$
    where $I_0(m)$ denotes the integral along the real edge, $I_\eta(m)$ denotes the integral along $\{{\rm{Im}}\, z=\eta\}$, and $I_\pm(m)$ denotes the integral over $\mathcal{C}_\pm$.

    \begin{figure}[H]
        \centering
        \includegraphics[width=0.8\linewidth]{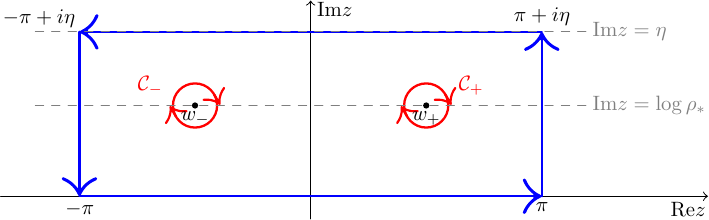}
        \caption{Rectangle Contour}
        \label{fig:rectangle-contour}
    \end{figure}

    For $I_\eta(m)$, we get
    \begin{equation}\label{eqn:I_eta(m)}
        |I_\eta(m)|\leq 2\pi \sup_{z\in[-\pi +i\eta,\pi+i\eta]} |\sigma(\cos z)| e^{-m\eta}.
    \end{equation}

    For $I_\pm(m)$, the residue theorem gives
    \begin{equation}\label{eqn:I_pm(m)-case-2}
        \begin{split}
            I_\pm(m) &= \oint_{\mathcal{C}_\pm} \sigma(\cos z)e^{imz}\,dz \\
            &= 2\pi i \cdot \frac{\mp 2\,\res(\sigma,z_\mp)}{\rho_\ast + \rho_\ast^{-1}}\rho_\ast^{-m} (\pm i)^m \\
            &= \frac{4\pi\,\res(\sigma,z_\mp)}{\rho_\ast + \rho_\ast^{-1}}(\pm i)^{m-1}\rho_\ast^{-m}, 
        \end{split}
    \end{equation}
    together with \eqref{eqn:I_eta(m)} and $\eta=\log R > \log\rho_\ast$, we obtain 
    $$ \widehat{\sigma}(m) = \frac{4\,\re(\res(\sigma,z_+)(-i)^{m-1})}{\rho_\ast + \rho_\ast^{-1}}\rho_\ast^{-m} + \mathcal O(R^{-m}), $$
    which completes the proof of \eqref{eqn:coeff-asymptotic}.

    Combining this identity with \eqref{eqn:coeff-asymptotic} proves \eqref{eqn:coeff-bound} in Case 2.
\end{proof}

    \begin{proof}[Proof of Theorem \ref{thm_coeff}]   
    The locations and residues of the singularities for the listed analytic functions are as follows.
    \begin{enumerate}
        \item $\sigma(t)=\tanh(t)$: singularities are $\{(\frac{\pi}{2}+k\pi)i:~k\in\ZZ\}$, hence we can take $z_\pm=\pm\frac{\pi i}{2}$, $\res(\sigma,z_\pm)=1$ and $R\in\Big(\frac{\pi}{2}+\sqrt{\frac{\pi^2}{4}+1},\frac{3\pi}{2}+\sqrt{\frac{9\pi^2}{4}+1}\Big)$. In this case $\rho_\ast=\rho_1=\frac{\pi}{2}+\sqrt{\frac{\pi^2}{4}+1}$.
        \item $\sigma(t) = \frac{1}{1+e^{-t}}$: singularities are $\{(\pi+2k\pi)i:~k\in\ZZ\}$, hence we can take $z_\pm=\pm\pi i$, $\res(\sigma,z_\pm)=1$ and $R\in\Big(\pi+\sqrt{\pi^2+1},3\pi+\sqrt{9\pi^2+1}\Big)$. In this case $\rho_\ast=\rho_2=\pi+\sqrt{\pi^2+1}$.
        \item $\sigma(t) = \frac{1}{1+t^2}$: singularities are $\pm i$, hence we can take $z_\pm = \pm i$, $\res(\sigma,z_\pm)=\mp\frac{i}{2}$ and $R\in (1+\sqrt{2},+\infty) $. In this case $\rho_\ast=\rho_3=1+\sqrt{2}$.
        \item $\sigma(t) = \frac{t}{1+t^2}$: singularities are $\pm i$, hence we can take $z_\pm = \pm i$, $\res(\sigma,z_\pm)=\frac{1}{2}$ and $R\in (1+\sqrt{2},+\infty) $. In this case $\rho_\ast=\rho_3=1+\sqrt{2}$.
    \end{enumerate}
        To characterize $E_\sigma$, firstly it suffices to notice $\frac{1}{1+t^2}$ is even,  $\tanh(t),\frac{t}{1+t^2}$ are odd, and $\frac{1}{1+e^{-t}}$ is odd up to a constant. Now it suffices to prove that, for the corresponding even/odd degrees $m$,
    \begin{equation}\label{eqn_sgn_sig_m}
        \mathrm{sign}(\widehat{\sigma}(m))=\mathrm{sign}(-I_+(m)-I_-(m)),
    \end{equation}
    where $-I_+(m)-I_-(m)$ is as in Theorem~\ref{thm:tanh-coeff-rate}. We consider three cases.

    \textbf{Case 1:} $d\geq 3$. In this case, $\lambda=\frac{d-1}{2}\geq 1$, therefore for $I_R(m)$, \eqref{eqn:I_R(m)-point-estimate} yields
    \begin{align*}
        |I_R(m)|&\leq \frac{c_{m,d}}{\pi} L(R)\,\sup_{z\in\mathcal{E}_R}\quad\left|\frac{ \sigma (z)}{(z+\sqrt{z^2-1})^{m+1}} {}_2F_1\left[\begin{matrix} m+1,\, 1-\lambda; \\ m+\lambda+1; \end{matrix}\, \frac{1}{(z+\sqrt{z^2-1})^2}\right]\right|\\
        & \leq \frac{c_{m,d}}{\pi} L(R)\,\sup_{z\in\mathcal{E}_R}| \sigma (z)|\,R^{-m-1}\frac{\Gamma(m+\lambda+1)}{\Gamma(m+1)\Gamma(\lambda)}\int_{0}^1 t^m(1-t)^{\lambda-1}\left|1-\frac{t}{(z+\sqrt{z^2-1})^2}\right|^{\lambda-1}\,dt \\
        &\leq \frac{c_{m,d}}{\pi} L(R)\, \sup_{z\in\mathcal{E}_R}| \sigma (z)|\,R^{-m-1}\frac{\Gamma(m+\lambda+1)}{\Gamma(m+1)\Gamma(\lambda)}\int_{0}^1 t^m(1-t)^{\lambda-1}(1+\frac{t}{R^2})^{\lambda-1}\,dt,
    \end{align*}
    where $L(R)$ denotes the circumference of $\mathcal{E}_R$. Therefore by \eqref{eqn:coeff-pm}, 
    \begin{equation}
        \begin{split}
        \frac{|I_R(m)|}{|I_+(m) + I_-(m)|} &\leq \frac{\frac{L(R)}{\pi}\sup_{z\in \mathcal{E}_R}|\sigma(z)| \int_0^1 t^m(1-t)^{\lambda-1}(1+\frac{t}{R^2})^{\lambda-1}\,dt}{4\,|\re(\res(\sigma,z_+)i^{m-1})|\int_0^1 t^m(1-t)^{\lambda-1}(1+\frac{t}{\rho_\ast^2})^{\lambda-1}\,dt} \lt(\frac{\rho_\ast}{R}\rt)^{m+1} \\
        &\leq \frac{L(R)\sup_{z\in \mathcal{E}_R}|\sigma(z)|}{4\pi\,|\re(\res(\sigma,z_+)i^{m-1})|}\cdot\lt(\frac{\rho_\ast}{R}\rt)^{m+1},
        \end{split}
    \end{equation}
    where the last inequality uses $R>\rho_\ast$. For the listed analytic functions, we can choose different $R$ to verify that the above quotient is less than $1$, which proves \eqref{eqn_sgn_sig_m} in this case.
    \begin{enumerate}
        \item $\sigma(t)=\tanh(t)$. Take $R=8$, so that $\mathcal{E}_R$ encloses only the two singularities $\pm\frac{\pi}{2}i$. We numerically obtain $$\sup_{z\in \mathcal{E}_R}|\sigma(z)|\approx 1.0212, \quad L(R)\approx 8.0004\,\pi,\quad \frac{\rho_\ast}{R} \approx 0.4291,$$
        therefore $|I_R(m)|<|I_+(m) + I_-(m)|$ for all $m\equiv 1\mod 2$.
        \item $\sigma(t)=\frac{1}{1+e^{-t}}$. Take $R=16$, so that $\mathcal{E}_R$ encloses only the two singularities $\pm\pi i$. We numerically obtain $$\sup_{z\in \mathcal{E}_R}|\sigma(z)|\approx 1.0006, \quad L(R)\approx 16.0006\,\pi,\quad \frac{\rho_\ast}{R} \approx 0.4129,$$
        therefore $|I_R(m)|<|I_+(m) + I_-(m)|$ for all $m\equiv 1\mod 2$.
        \item $\sigma(t)=\frac{1}{1+t^2}$. Take $R=5$. We numerically obtain $$\sup_{z\in \mathcal{E}_R}|\sigma(z)|\approx 0.2101, \quad L(R)\approx 5.0016\,\pi,\quad \frac{\rho_\ast}{R} \approx 0.4818,$$
        therefore $|I_R(m)|<|I_+(m) + I_-(m)|$ for all $m\equiv 0\mod 2$.
        \item $\sigma(t)=\frac{t}{1+t^2}$. Take $R=5$. We numerically obtain $$\sup_{z\in \mathcal{E}_R}|\sigma(z)|\approx 0.5042, \quad L(R)\approx 5.0016\,\pi,\quad \frac{\rho_\ast}{R} \approx 0.4818,$$
        therefore $|I_R(m)|<|I_+(m) + I_-(m)|$ for all $m\equiv 1\mod 2$.
    \end{enumerate}

    \textbf{Case 2}: $d=2$. In this case, $\lambda=\frac{d-1}{2}=\frac{1}{2}<1$, therefore for $I_R(m)$, by \eqref{eqn:I_R(m)-point-estimate}
    \begin{align*}
        |I_R(m)|&\leq \frac{c_{m,d}}{\pi} L(R)\,\sup_{z\in\mathcal{E}_R}| \sigma (z)|\,R^{-m-1}\cdot (1-R^{-2})^{-\frac{1}{2}}.
    \end{align*}
    Similarly from \eqref{eqn:coeff-pm} we can get 
    $$|I_+(m) + I_-(m)|\geq 4c_{m,d}\,|\re(\res(\sigma,z_+)i^{m-1})|\,\rho_\ast^{-(m+1)}\,(1+\frac{1}{\rho_\ast^2})^{-\frac{1}{2}}. $$
    Therefore, 
    \begin{equation}
        \begin{split}
        \frac{|I_R(m)|}{|I_+(m) + I_-(m)|} &\leq \frac{L(R)\,\sup_{z\in \mathcal{E}_R}|\sigma(z)| }{4\pi\,|\re(\res(\sigma,z_+)i^{m-1})|}\,\sqrt{\frac{1+\frac{1}{\rho_\ast^2}}{1-\frac{1}{R^2}}} \lt(\frac{\rho_\ast}{R}\rt)^{m+1} \\
        &\leq \frac{L(R)\sup_{z\in \mathcal{E}_R}|\sigma(z)|}{4\pi\,|\re(\res(\sigma,z_+)i^{m-1})|}\cdot\sqrt{\frac{1+\frac{1}{\rho_\ast^2}}{1-\frac{1}{R^2}}}\lt(\frac{\rho_\ast}{R}\rt)^{m+1}.
        \end{split}
    \end{equation}
    Taking the same values of $R$ as in Case 1, we verify that the above quotient is strictly less than $1$ for the listed analytic functions, which proves \eqref{eqn_sgn_sig_m} in this case.

    \textbf{Case 3:} $d=1$.
    Since $\eta = \log R$ and $\cos z$ maps $[-\pi +i\eta,\pi+i\eta]$ to $\mathcal{E}_R$, we have $$\sup_{z\in[-\pi +i\eta,\pi+i\eta]} |\sigma(\cos z)|=\sup_{z\in \mathcal{E}_R}|\sigma(z)|.$$
    Therefore, \eqref{eqn:I_eta(m)} and \eqref{eqn:I_pm(m)-case-2} yield
    $$ \frac{|I_\eta(m)|}{|I_+(m) + I_-(m)|}\leq \frac{\rho_\ast + \rho_\ast^{-1}}{4\,|\re(\res(\sigma,z_+)i^{m-1})|}\lt(\frac{\rho_\ast}{R}\rt)^m\sup_{z\in[-\pi +i\eta,\pi+i\eta]} |\sigma(\cos z)|.$$
    Taking the same values of $R$ as in Case 1, we verify that the above quotient is strictly less than $1$, which proves \eqref{eqn_sgn_sig_m} in this case.

    Finally, we prove \eqref{eqn_arctan_coeff}. We use the differentiation formula for Gegenbauer polynomials (see, e.g., \cite[(4.7.14)]{szego1975orthogonal}):
    \begin{equation}\label{eqn_der_Gegen}
        \frac{d}{dt}C_m^{(\lambda)}(t)=2\lambda C_{m-1}^{(\lambda+1)}(t).
    \end{equation}
    Thus we can combine \eqref{eqn_der_Gegen} with \eqref{eqn:coeff-relation} and rewrite the expansion
    \begin{equation*}
        \begin{split}
            \frac{d}{dt}\arctan(t)=\frac{d}{dt}\sul_{m=0}^\infty\widehat{\sigma}(m)p_m(t)=\frac{d}{dt}\sul_{m=0}^\infty\frac{\widehat{\sigma}(m)}{q_{m,d}}C_m^{(\lambda)}(t)=\sul_{m=1}^\infty\frac{\widehat{\sigma}(m)}{q_{m,d}}2\lambda C_{m-1}^{(\lambda+1)}(t)
        \end{split}
    \end{equation*}
    On the other hand, replacing $d$ by $d+2$ in \eqref{eqn_coef_est_3}, we can write
    \begin{equation*}
            \frac{1}{1+t^2}=\sul_{m=0}^\infty a_{d+2}(m)q_{m,d+2}^{-1}C_m^{(\lambda+1)}(t),
    \end{equation*}
    where $a_{d+2}(m)$ are the coefficients in \eqref{eqn_coef_est_3} with $d$ replaced by $d+2$:
        \begin{equation*}
        \begin{split}
            &a_{d+2}(m)\simeq (-1)^{\frac{m}{2}}m^{-\frac{d+1}{2}}\rho_3^{-m},\qquad m\equiv0\mod 2,\\
            &a_{d+2}(m)=0,\qquad m\equiv1\mod 2,
        \end{split}
    \end{equation*}
    Comparing these two expansions, we have
    \begin{equation*}
        \frac{\widehat{\sigma}(m+1)}{q_{m+1,d}}2\lambda=a_{d+2}(m)q_{m,d+2}^{-1},\quad m\in\NN_0
    \end{equation*}
    and consequently
    \begin{equation*}
        \begin{split}
            &\widehat{\sigma}(m)\simeq (-1)^{\frac{m-1}{2}}m^{-\frac{d+1}{2}}\rho_3^{-m},\qquad m\equiv1\mod 2,\\
            &\widehat{\sigma}(m)=0,\qquad m\equiv0\mod 2.
        \end{split}
    \end{equation*}
    \end{proof}

    \begin{remark}
        The approach of Theorem~\ref{thm:tanh-coeff-rate} can also be adapted to cases in which $z_\pm$ lie on $\RR$. Suppose that $z_+>0$ and $z_-=-z_+$; otherwise, one may take $R\in(\min(|z_-|,|z_+|),\max(|z_-|,|z_+|))$ so that $\mathcal{E}_R$ encloses only one singularity. The same argument then gives estimates analogous to those in Theorem~\ref{thm:tanh-coeff-rate}.
    \end{remark}

This completes the proof of the coefficient results. We now turn to the approximation theorems, whose proofs use the preceding analysis only through the abstract estimate \eqref{eqn_decay_thm1}; no further information about the pole locations or residues is required.

\subsection{Proofs of the approximation-rate theorems}
The proofs of Theorems~\ref{thm_main_sob} and \ref{thm:appr_rate_sph} use exact low-frequency matching through the positive quadrature formula in Lemma~\ref{lem:quadrature}, taken from \cite{narcowich2006localized}. The overall approximation strategy is inspired by ideas in \cite{mhaskar1999approximation}, while \cite{liu2025achieving} provides related Hilbert-space and matrix estimates for ReLU$^k$. Spherical harmonics up to a degree determined by the mesh norm are matched exactly, and the remaining high-frequency contribution is estimated using \eqref{eqn_decay_thm1}. The conclusions are the $\mathcal H^s$ estimate and explicit normalized $\ell^2$ control of the outer coefficients for the coefficient structures considered here.

We begin by observing from \eqref{eqn_fn_expans} that the spherical coefficients of $f_n=\sul_{j=1}^na_j\sigma(\theta_j^*\cdot\circ)$ are
\begin{equation}\label{eqn:hat_fn_explicit}
    \widehat{f_n}(m,\ell)=\Big(\sul_{j=1}^na_jY_{m,\ell}(\theta_j^*)\Big)\widehat{\sigma}(m).
\end{equation}

By Lemma~\ref{lem:quadrature}, there exist nonnegative numbers
\begin{equation}\label{eqn:nu_bound}
    \tau_1,\dots,\tau_n\lesssim h^d
\end{equation}
such that
\begin{equation}\label{eqn:ideal_aj_tofind}
    \begin{split}
    \widehat f(m,\ell)=&\fint_{\SS^d}f(\eta)Y_{m,\ell}(\eta)d\eta=\fint_{\SS^d}(\Pi_{m}f)(\eta)Y_{m,\ell}(\eta)d\eta\\
    =&\fint_{\SS^d}\Bigg(\sul_{\substack{\Pi_{m'}f\neq0\\m'\leq J}}\widehat{\sigma}(m')^{-1}(\Pi_{m'}f)(\eta)\Bigg)\widehat{\sigma}(m)Y_{m,\ell}(\eta)d\eta\\
    =&\sul_{j=1}^n\tau_j\Bigg(\sul_{\substack{\Pi_{m'}f\neq0\\m'\leq J}}\widehat{\sigma}(m')^{-1}(\Pi_{m'}f)(\theta_j^*)\Bigg)\widehat{\sigma}(m)Y_{m,\ell}(\theta_j^*),\qquad m\leq J.
    \end{split}
\end{equation}
where $J$ is the integer in Lemma~\ref{lem:quadrature}. Under Assumption~\ref{assum_nece}, \eqref{eqn:ideal_aj_tofind} is well-defined.

Comparing \eqref{eqn:hat_fn_explicit} and \eqref{eqn:ideal_aj_tofind}, by taking
\begin{equation}
    a_j=\tau_j\Bigg(\sul_{\substack{\Pi_{m}f\neq0\\m\leq J}}\widehat{\sigma}(m)^{-1}(\Pi_mf)(\theta_j^*)\Bigg),\qquad j=1,\dots,n,
\end{equation}
we have
\begin{equation}\label{eqn:trunc_equal}
    \widehat {f_n}(m,\ell)=\widehat f(m,\ell),\qquad m\leq J.
\end{equation}

Given \eqref{eqn:trunc_equal}, in terms of \eqref{eqn:Sob_norm_Parseval} we can write
\begin{equation}\label{eqn_est_f-fn}
    \begin{split}
        \lt\|f-f_n\rt\|_{\mathcal{H}^\ss(\SS^d)}^2
        =&\sul_{m=0}^\infty\sul_{\ell=1}^{N(m)}\lt(\widehat f(m,\ell)-\widehat{f_n}(m,\ell)\rt)^2(m^{2\ss}+1)\\
        =&\sul_{m=J+1}^\infty\sul_{\ell=1}^{N(m)}\lt(\widehat f(m,\ell)-\widehat{f_n}(m,\ell)\rt)^2(m^{2\ss}+1)\\
        \leq&4(I_1+I_2),
    \end{split}
\end{equation}
where
\begin{equation*}
    \begin{split}
        I_1=\sul_{m=J+1}^\infty\sul_{\ell=1}^{N(m)}\widehat f(m,\ell)^2m^{2\ss},\qquad
        I_2=\sul_{m=J+1}^\infty\sul_{\ell=1}^{N(m)}\widehat {f_n}(m,\ell)^2m^{2\ss},
    \end{split}
\end{equation*}
Before we proceed, we give a preliminary study of $I_2$. By \eqref{eqn:hat_fn_explicit}, we have
\begin{equation}\label{eqn:explicit_I2}
    \begin{split}
        I_2=&\sul_{m=J+1}^\infty m^{2\ss}\widehat{\sigma}(m)^2\sul_{\ell=1}^{N(m)}\Big(\sul_{j=1}^na_jY_{m,\ell}(\theta_j^*)\Big)^2=\sul_{m=J+1}^\infty m^{2\ss}\widehat{\sigma}(m)^2\sul_{\ell=1}^{N(m)}\sul_{1\leq i,j\leq n}a_ia_jY_{m,\ell}(\theta_i^*)Y_{m,\ell}(\theta_j^*)\\
        =&\sul_{m=J+1}^\infty m^{2\ss}\widehat{\sigma}(m)^2\sul_{1\leq i,j\leq n}a_ia_jp_m(\theta_i^*\cdot\theta_j^*)
        =\sul_{m=J+1}^{\infty} m^{2\ss}\widehat{\sigma}(m)^2\sul_{1\leq i,j\leq n}\frac{a_ia_j}{\sqrt{\tau_i\tau_j}}\sqrt{\tau_i\tau_j}p_m(\theta_i^*\cdot\theta_j^*)\\
        =&\sul_{m=J+1}^{\infty} m^{2\ss}\widehat{\sigma}(m)^2\tilde a^\top Q(m)\tilde a
    \end{split}
\end{equation}
where
$$\tilde a^\top=\Big(\frac{a_1}{\sqrt{\tau_1}},\dots,\frac{a_n}{\sqrt{\tau_n}}\Big),$$
and $\{Q(m)\}_{m=0}^\infty$ are the matrices given by
\begin{equation}
    (Q(m))_{i,j}=\sqrt{\tau_i\tau_j}p_m(\theta_i^*\cdot\theta_j^*),\quad i,j=1,\dots,n,
\end{equation}

By \eqref{eqn_decay_thm1}, the norm of $a$ is estimated as
\begin{equation}\label{eqn_est_a}
    \begin{split}
        \|\tilde a\|_2^2=&\sul_{j=1}^n\tau_j\Bigg(\sul_{\substack{\Pi_{m}f\neq0\\m\leq J}}\widehat{\sigma}(m)^{-1}(\Pi_mf)(\theta_j^*)\Bigg)^2=\fint_{\SS^d}\Bigg(\sul_{\substack{\Pi_{m}f\neq0\\m\leq J}}\widehat{\sigma}(m)^{-1}(\Pi_mf)(\eta)\Bigg)^2d\eta\\
        =&\sul_{\substack{\Pi_{m}f\neq0\\m\leq J}}\widehat{\sigma}(m)^{-2}\sul_{\ell=1}^{N(m)}\widehat{f}(m,\ell)^2\simeq \sul_{\substack{\Pi_{m}f\neq0\\m\leq J}}m^{-2\alpha}\rho^{2m}\sul_{\ell=1}^{N(m)}\widehat{f}(m,\ell)^2.
    \end{split}
\end{equation}
Using $|p_m(\theta_i^*\cdot\theta_j^*)|\leq p_m(1)=N(m)\simeq m^{d-1}$ and \eqref{eqn:nu_bound}, we have
\begin{equation}\label{eqn_Qm_norm1}
    \begin{split}
        \|Q(m)\|_2^2\leq&\Big(\mal_{1\leq i\leq n}\sul_{j=1}^n|\sqrt{\tau_i\tau_j}p_m(\theta_i^*\cdot\theta_j^*)|\Big)\Big(\mal_{1\leq j\leq n}\sul_{i=1}^n|\sqrt{\tau_j\tau_i}p_m(\theta_j^*\cdot\theta_i^*)|\Big)
        \lesssim m^{2(d-1)}.
    \end{split}
\end{equation}
In particular, for $m\leq2J$, we can write $Q(m)$ as
\begin{equation*}
    Q(m)=\Big(\sqrt{\tau_i\tau_j}\sul_{\ell=1}^{N(m)}Y_{m,\ell}(\theta_i^*)Y_{m,\ell}(\theta_j^*)\Big)_{i,j=1}^n=\mathcal{Y}(m)\mathcal{Y}(m)^\top,
\end{equation*}
where $\mathcal{Y}(m)$ is a $n\times N(m)$ matrix given by
\begin{equation*}
    (\mathcal{Y}(m))_{j,\ell}=\sqrt{\tau_j}Y_{m,\ell}(\theta_j^*),\quad 1\leq j\leq n,~1\leq\ell\leq N(m).
\end{equation*}
Then $\mathcal{Y}(m)^\top\mathcal{Y}(m)$ is an identity $N(m)\times N(m)$ matrix:
\begin{equation*}
    \begin{split}
        \mathcal{Y}(m)^\top\mathcal{Y}(m)=\Big(\sul_{j=1}^n\sqrt{\tau_j}Y_{m,\ell}(\theta_j^*)\sqrt{\tau_j}Y_{m,\ell'}(\theta_j^*)\Big)_{\ell,\ell'=1}^{N(m)}=\Big(\fint_{\SS^d}Y_{m,\ell}(\eta)Y_{m,\ell'}(\eta)d\eta\Big)_{\ell,\ell'=1}^{N(m)}=I
    \end{split}
\end{equation*}
and
\begin{equation}\label{eqn_Qm_norm2}
    \begin{split}
        \|Q(m)\|_2=\|\mathcal{Y}(m)\mathcal{Y}(m)^\top\|_2=\|\mathcal{Y}(m)^\top\mathcal{Y}(m)\|_2=1.
    \end{split}
\end{equation}

The preceding low-frequency matching construction and matrix estimates are common to both regularity regimes. We first complete the Sobolev estimate.

\begin{proof}[Proof of Theorem \ref{thm_main_sob}]
We divide \eqref{eqn:explicit_I2} into two parts. For $m\geq2J+1$, \eqref{eqn_decay_thm1}, \eqref{eqn_est_a}, and \eqref{eqn_Qm_norm1} give
\begin{equation}\label{eqn:est_I2_1}
    \begin{split}
        &\sul_{m=2J+1}^\infty m^{2\ss}\widehat{\sigma}(m)^2\tilde a^\top Q(m)\tilde a\lesssim\sul_{m=2J+1}^\infty m^{2\ss+2\alpha+d-1}\rho^{-2m}\|\tilde a\|_2^2\\
        \lesssim&\sul_{m=2J+1}^\infty m^{2\ss+2\alpha+d-1}\rho^{-2m}\Big(\sul_{\substack{\Pi_{m'}f\neq0\\m'\leq J}}\frac{m'^{-2\alpha}\rho^{2m'}}{m'^{2r}+1}\sul_{\ell=1}^{N(m')}(m'^{2r}+1)\widehat{f}(m',\ell)^2\Big)\\
        \lesssim&\sul_{m=2J+1}^\infty J^{-2\alpha-2r}m^{2\ss+2\alpha+d-1}\rho^{-2(m-J)}\|f\|_{\mathcal{H}^r(\SS^d)}^2\lesssim J^{-2\rr}\|f\|_{\mathcal{H}^r(\SS^d)}^2.
    \end{split}
\end{equation}
For $m\leq2J$, \eqref{eqn_decay_thm1} and \eqref{eqn_Qm_norm2} give
\begin{equation}\label{eqn:est_I2_2}
    \begin{split}
        &\sul_{m=J+1}^{2J} m^{2\ss}\widehat{\sigma}(m)^2\tilde a^\top Q(m)\tilde a\lesssim\sul_{m=J+1}^{2J} m^{2\ss+2\alpha}\rho^{-2m}\|\tilde a\|_2^2\\
        \lesssim&\sul_{m=J+1}^{2J} m^{2\ss+2\alpha}\rho^{-2m}\Big(\sul_{\substack{\Pi_{m'}f\neq0\\m'\leq J}}m'^{-2\alpha}\rho^{2m'}\sul_{\ell=1}^{N(m')}\widehat{f}(m',\ell)^2\Big)\\
        \leq&\sul_{m=J+1}^{2J} m^{2\ss+2\alpha}\rho^{-2m}\rho^{2J}J^{-2\alpha-2\rr}\Big(\sul_{\substack{\Pi_{m'}f\neq0\\m'\leq J}}m'^{2\rr}\sul_{\ell=1}^{N(m')}\widehat{f}(m',\ell)^2\Big)\\
        \lesssim& J^{2s-2r}\|f\|_{\mathcal{H}^r(\SS^d)}^2.
    \end{split}
\end{equation}

Substituting \eqref{eqn:est_I2_1} and \eqref{eqn:est_I2_2} into \eqref{eqn:explicit_I2}, we have
\begin{equation}
    I_2\lesssim J^{2s-2r}\|f\|_{\mathcal{H}^r(\SS^d)}^2\lesssim n^{-\frac{2r-2s}{d}}\|f\|_{\mathcal{H}^r(\SS^d)}^2.
\end{equation}
On the other hand, since $f\in\mathcal{H}^r(\SS^d)$,
\begin{equation}\label{eqn_est_I1}
    I_1\leq J^{2\ss-2\rr}\sul_{m=J+1}^\infty\sul_{\ell=1}^{N(m)}\widehat f(m,\ell)^2m^{2\rr}\lesssim n^{-\frac{2\rr-2\ss}{d}}\|f\|_{\mathcal{H}^r(\SS^d)}^2.
\end{equation}
Substituting these estimates into \eqref{eqn_est_f-fn} gives
\begin{equation}
    \|f-f_n\|_{\mathcal{H}^s(\SS^d)}\lesssim n^{-\frac{r-s}{d}}\|f\|_{\mathcal{H}^r(\SS^d)}.
\end{equation}
Finally, by recalling $\tau_j\simeq n^{-1}$ and \eqref{eqn_est_a}, we have
\begin{equation*}
    \Big(n\sul_{j=1}^na_j^2\Big)\simeq\|\tilde a\|_2^2\lesssim J^{-2\alpha}\rho^{2J}\sul_{\substack{\Pi_{m}f\neq0\\m\leq J}}\sul_{\ell=1}^{N(m)}\widehat{f}(m,\ell)^2\leq J^{-2\alpha}\rho^{2J}\|f\|_{\mathcal{L}^2(\SS^d)}^2\lesssim\rho^{2C_3\sqrt[d]{n}}\|f\|_{\mathcal{L}^2(\SS^d)}^2
\end{equation*}
for some $C_3>0$, which proves \eqref{eqn_coef_bound_thm}.

\end{proof}

For targets in $\mathrm{Hol}_{\rho_0}(\SS^d)$, the same exact low-frequency matching and matrix estimates apply; only the tail estimates are different.

\begin{proof}[Proof of Theorem \ref{thm:appr_rate_sph}]
Again, by \eqref{eqn:explicit_I2}, \eqref{eqn_Qm_norm1},
\begin{equation}\label{eqn:est_I2_3}
    \begin{split}
        I_2=&\sul_{m=J+1}^\infty m^{2\ss}\widehat{\sigma}(m)^2\tilde a^\top Q(m)\tilde a
        \lesssim\sul_{m=J+1}^\infty m^{2\ss+d-1}\rho^{-2m}\Big(\sul_{\substack{\Pi_{m'}f\neq0\\m'\leq J}}m'^{-2\alpha}\rho^{2m'}\sul_{\ell=1}^{N(m')}\widehat{f}(m',\ell)^2\Big)\\
        \leq&\sul_{m=J+1}^\infty m^{2\ss+d-1}\rho^{-2m}J^{-2\alpha}\Big(\frac{\rho}{\rho_0}\Big)^{2J}\Big(\sul_{\substack{\Pi_{m'}f\neq0\\m'\leq J}}\rho_0^{2m'}\sul_{\ell=1}^{N(m')}\widehat{f}(m',\ell)^2\Big)\\
        \lesssim& \ \rho_0^{-2J}J^{2\ss+d-1-2\alpha}\|f\|_{\mathrm{Hol}_{\rho_0}(\SS^d)}^2\lesssim \rho_0^{-J}\|f\|_{\mathrm{Hol}_{\rho_0}(\SS^d)}^2.
    \end{split}
\end{equation}


On the other hand, since $f\in\mathrm{Hol}_{\rho_0}(\SS^d)$,
\begin{equation}\label{eqn_est_I1_2}
    I_1\lesssim \rho_0^{-2J}J^{2\ss}\sul_{m=J+1}^\infty\sul_{\ell=1}^{N(m)}\widehat f(m,\ell)^2 \rho_0^{2m}\lesssim \rho_0^{-J}\|f\|_{\mathrm{Hol}_{\rho_0}(\SS^d)}^2.
\end{equation}
Substituting these estimates into \eqref{eqn_est_f-fn} gives
\begin{equation}
    \|f-f_n\|_{\mathcal{H}^s(\SS^d)}\lesssim \rho_0^{-\frac{J}{2}}\|f\|_{\mathrm{Hol}_{\rho_0}(\SS^d)} \leq \rho_0^{-C_4\sqrt[d]{n}}\|f\|_{\mathrm{Hol}_{\rho_0}(\SS^d)},
\end{equation}
where $J=\lfloor C_2 h^{-1}\rfloor$ as in Lemma \ref{lem:quadrature} and $C_4$ is chosen satisfying $C_4 \sqrt[d]{n}\leq J/2$.

The coefficients are the same as those in Theorem~\ref{thm_main_sob}.
\end{proof}

\section{Concluding remarks}\label{sec_concl}

This paper develops the spectral information needed to obtain refined spherical approximation estimates for several analytic activations. Starting from the contour representation of ultraspherical coefficients, we derive an asymptotic formula with explicit residue dependence for a conjugate pair of simple poles on the imaginary axis. Besides identifying the exponential scale and oscillatory phase, the formula gives a matching lower bound for the ultraspherical coefficients at the degrees where they do not vanish and determines the relevant parity subsequence. It applies directly to $\tanh$, the logistic sigmoid, and the rational activations considered here; the $\arctan$ estimate follows through differentiation.

The approximation argument is inspired by the earlier analytic-activation results in \cite{mhaskar1999approximation} and the recent work \cite{liu2025achieving}. Under an abstract two-sided estimate for the nonzero ultraspherical coefficients, we obtain the corresponding approximation estimate with an $\mathcal H^s(\SS^d)$ norm on the left-hand side and simultaneously obtain explicit normalized $\ell^2$ control of the outer coefficients. The spectral results verify the required hypothesis for the oscillatory and parity-restricted analytic activations treated in the paper. The same construction also gives the stated analytic estimate for targets in $\mathrm{Hol}_{\rho_0}(\SS^d)$.

Together, these results give a route from complex singularities to spectral asymptotics and lower bounds at the degrees where the ultraspherical coefficients do not vanish, and from this spectral information to Sobolev-norm approximation and explicit normalized $\ell^2$ coefficient control. Natural directions for further work include more general configurations of complex singularities, a direct treatment of branch-point singularities, and approximation lower bounds for possible saturation behavior on the analytic scale. Finite-bit representations of smooth functions on the sphere have been studied in \cite{mhaskar2005representation}; relating the coefficient bounds obtained here to sharper stability and quantization estimates remains an interesting problem.

\bibliographystyle{abbrv}
	\bibliography{ref}

\end{document}